\documentclass[reqno, 12pt]{amsart}
\usepackage{amsmath, amsthm, amsfonts, amssymb}
\usepackage{color,xcolor}

\newtheorem{theorem}{Theorem}[section]
\newtheorem{lemma}[theorem]{Lemma}

\newtheorem{corollary}[theorem]{Corollary}
\theoremstyle{definition}
\newtheorem{definition}[theorem]{Definition}

\theoremstyle{remark}

\numberwithin{equation}{section}

 \numberwithin{equation}{section}

 \allowdisplaybreaks
 
\begin{document}

\title{Results on normal harmonic and $\varphi$-normal harmonic mappings}
\author[N. Bharti]{Nikhil Bharti}
\address{
\begin{tabular}{lll}
&Nikhil Bharti\\
&Department of Mathematics\\
&University of Jammu\\
&Jammu-180006\\ 
&India\\
\end{tabular}}
\email{nikhilbharti94@gmail.com}

\author[N. V. Thin]{Nguyen Van Thin$^{*}$}
\address{
\begin{tabular}{lll}
&Nguyen Van Thin\\
&Department of Mathematics\\
&Thai Nguyen University of Education\\
&Luong Ngoc Quyen street\\
&Thai Nguyen City, Thai Nguyen\\ 
&Vietnam\\
\end{tabular}}
\email{thinmath@gmail.com}
\thanks{$^*$ Corresponding author: Nguyen Van Thin}

\begin{abstract}
In this paper, we study the concepts of normal functions and 
$\varphi$-normal functions in the framework of planar harmonic mappings. We establish the harmonic mapping counterpart of the well-known Zalcman-Pang lemma and as a consequence, we prove that a harmonic mapping whose spherical derivative is bounded away from zero is normal. Furthermore, we introduce the concept of the extended spherical derivative for harmonic mappings and obtain several sufficient conditions for a harmonic mapping to be $\varphi$-normal. 
\end{abstract}

\maketitle

\vspace{-0.6cm}
\section{Introduction and Main Results }\label{S:intro}

The concept of normal harmonic mappings is closely connected to the concept of normal functions, which in turn relates to Montel's theory of normal families (see \cite{schiff}). The first ideas on the concept of normal functions were attributed to Yosida \cite{yosida} and later explored by Noshiro \cite{noshiro}, though neither author used the term "normal function." The term normal function was coined by Lehto and Virtanen in their seminal paper \cite{lehto}, where they highlighted the relevance of normal functions to problems related to the boundary behavior of meromorphic functions. To state formally, a meromorphic function $f$ defined on the open unit disk $\mathbb{D}$ is said to be normal in $\mathbb{D}$ if $\mathcal{F}=\left\{f\circ \sigma: \sigma\in \text{Aut}(\mathbb{D})\right\}$ forms a normal family in $\mathbb{D},$ where $\text{Aut}(\mathbb{D})$ is the group of conformal automorphisms of $\mathbb{D}.$ This definition is due to Noshiro \cite[p. 149]{noshiro}, however, it can be broadened to apply to any simply connected domain (see \cite[p. 53]{lehto}). For a detailed account on normal functions and normal families as well, we refer the reader to the recent book by Dovbush and Krantz \cite{dovbush}. 

\smallskip

Recently, there has been growing interest in studying the concept of normal functions in the context of planar harmonic mappings, driven by the broad applicability of harmonic mappings in geometric function theory (see \cite{ahamed, arbelaez, bohra, deng}) and also due to their correlation with ideal fluid flow problems (see \cite{aleman, constantin}). In this paper, our primary objective is to conduct further investigation into the study of normal harmonic mappings, extending well-known results of meromorphic functions to  harmonic mappings from the unit disk $\mathbb{D}$ to the complex plane.  

\smallskip

Before we state our results, let us recall some basic notions and results of planar harmonic mappings.
A complex-valued function $f$ defined on a simply connected domain $D\subseteq\mathbb{C}$ is said to be a  harmonic mapping if it has a canonical decomposition of the form $f=h+\bar{g},$ where $h$ and $g$ are holomorphic in $D$ and $g(z_0)=0$ for some point $z_0\in D$ (see \cite{duren}). It is a long established result of Lewy \cite{lewy} that a harmonic mapping $f=h+\bar{g}$ is locally univalent and sense-preserving in $\mathbb{D}$ if and only if $h'(z)\neq 0$ and the Jacobian $J_{f}(z)>0,$ where $J_{f}(z)=|h'(z)|^2-|g'(z)|^2.$ Motivated by the work of Colonna \cite{colonna} on Bloch harmonic functions, Arbel\'{a}ez et al. \cite{arbelaez}, in $2019,$ introduced the idea of normal harmonic mappings. In brief, a harmonic mapping $f$ is said to be normal in $\mathbb{D}$ if it satisfies the following 
 Lipschitz-type condition: $$\sup\limits_{\substack{z\neq w \\ z, w~\in~\mathbb{D}}}\frac{\sigma(f(z), f(w))}{\Lambda_{h}(z, w)}<\infty,$$ where $\Lambda_{h}(z, w)$ is the hyperbolic distance between $z$ and $w$ given by $$\Lambda_{h}(z,w)=\frac{1}{2}\log\left(\frac{1+t}{1-t}\right), ~t=\left|\frac{z-w}{1-\bar{w}z}\right|.$$ Arbel\'{a}ez et al. \cite{arbelaez} also gave the following equivalent and more practicable characterization of normal harmonic mappings, which we adopt as our definition.
\smallskip

\begin{definition}\label{D: normal harmonic}\cite[Proposition 2.1]{arbelaez}
A harmonic mapping $f=h+\bar{g}$ in $\mathbb{D}$ is said to be {\it normal} in $\mathbb{D}$ if 
\begin{equation}\label{eq: hnm1}
\sup_{z\in \mathbb{D}} (1-|z|^2)f^{\#}(z)<\infty,
\end{equation}
 where 
 \begin{equation}\label{eq: spherical derivative}
 f^{\#}(z)=\frac{|h'(z)|+|g'(z)|}{1+|f(z)|^2}
 \end{equation}
 \end{definition}
 It is worthwhile to mention that, for a function $F,$ meromorphic on a domain $D\subseteq\mathbb{C},$ the quantity $F^\#(z)$ given by $F^\#(z)=|F'(z)|/(1+|F(z)|^2),$ is called the spherical derivative of $F$ (see \cite[p. 197]{marty}). We retain this terminology for the harmonic mapping $f$ and call the quantity $f^\#(z),$ defined in \eqref{eq: spherical derivative}, the spherical derivative of $f.$

\medskip

In this paper, we present several significant results that not only generalize previous findings but also introduce novel insights that, to the best of our knowledge, have not been explored before. Continuing the investigations as in \cite{arbelaez, deng}, we first establish a necessary and sufficient condition for a harmonic mapping in $\mathbb{D}$ to be normal, serving as the harmonic counterpart to the widely recognized Zalcman-Pang lemma (see \cite{chen-gu, hua, pang-zalcman}). This is stated as: 

\begin{theorem}\label{thm:zp} 
A non-constant mapping $f$ harmonic in $\mathbb{D}$ is normal if and only if there do not 
exist sequences of points $\{z_n\}\subset\mathbb{D}$ and of real numbers $\{\rho_n\}$ with $\rho_n>0$ and $\rho_n\to 0$ as $n\to\infty$ such that the functions   
$$g_n(\zeta);= \rho_n^{\alpha}f(z_n+ \rho_n \zeta)$$ converges locally uniformly in $\mathbb{C}$ to a non-constant harmonic mapping $g(\zeta)$ satisfying $g^{\#}(\zeta)\leq g^{\#}(0)=1,$ where $\alpha\in (-1,\infty).$
\end{theorem}
We remark that if we take $\alpha=0$ in Theorem \ref{thm:zp}, then we recover, as a corollary, a result of Deng et al. \cite[Theorem 1]{deng}, which is a harmonic mapping analogue of the rescaling result of Lohwater and Pommerenke \cite[Theorem 1]{lohwater}.

\begin{corollary} \cite[Theorem 1]{deng}\label{cor:deng}
A non-constant mapping $f$ harmonic in $\mathbb{D}$ is normal if and only if there do not 
exist sequences $\{z_n\}\subset\mathbb{D},~\{\rho_n\},~\rho_n>0,~\rho_n\to 0$ as $n\to\infty$ such that the functions $f(z_n+ \rho_n \zeta)$ converges locally uniformly in $\mathbb{C}$ to a non-constant harmonic mapping.
\end{corollary}

\medskip

Deng et al. \cite[Lemma 1]{deng} also proved an analogue of Marty's theorem \cite[Theorem 4]{marty} for a class of harmonic mappings which states that a family $\mathcal{F}$ of harmonic mappings $f=h+\bar{g}$ in $\mathbb{D}$ is normal if the corresponding family of spherical derivatives $\{f^\#: f\in\mathcal{F}\}$ is locally uniformly bounded. The converse holds if the real part of $h(z)g(z)>0$ (see \cite[Lemma 2.1]{bohra}). Therefore, for a harmonic mapping $f=h+\bar{g}$ in $\mathbb{D},$ the boundedness of the spherical derivative is sufficient to ensure normality of $f$ in $\mathbb{D}.$ As an application of Theorem \ref{thm:zp}, we obtain the following result: 

\begin{theorem}\label{thm:g}
    Let $f=h+\bar{g}$ be a harmonic mapping in $\mathbb{D}$ and let $\epsilon>0$ be a real number. If, for each $z\in\mathbb{D},$ $f^{\#}(z)>\epsilon,$ then $f$ is a normal harmonic mapping.
\end{theorem}

Theorem \ref{thm:g} essentially demonstrates that normality of a harmonic mapping $f$ in $\mathbb{D}$ can be deduce even when the spherical derivative of $f$ is bounded away from zero. This is in consonance with the result of Grahl and Nevo \cite[Theorem 1]{grahl} for meromorphic functions. Thus, Theorem \ref{thm:g} serves as the harmonic mapping analogue of the result of Grahl and Nevo.

\medskip

Following the work of Aulaskari and R\"{a}tty\"{a} \cite{aulaskari-1} on $\varphi$-normal meromorphic functions, Bohra et al. \cite{bohra} recently  extended the concept of $\varphi$-normal functions to planar harmonic mappings. Our next objective is to make further explorations in the newly established concept of $\varphi$-normal harmonic mappings. 

\medskip

We begin by recalling the definition of a $\varphi$-normal function as introduced by Aulaskari and  R\"{a}tty\"{a} \cite[Definition 1]{aulaskari-1}.  
An function $\varphi:[0,1) \rightarrow(0, \infty)$ is termed {\it smoothly increasing} if

\begin{equation}\label{eq:phi}
\varphi(r)(1-r) \rightarrow \infty \quad \text { as } r \rightarrow 1^{-}
\end{equation}
and
\begin{equation}\label{eq: phi cnvrgnc}
\mathcal{R}_{a}(z):=\frac{\varphi(|a+z / \varphi(|a|)|)}{\varphi(|a|)} \rightarrow 1 \quad \text { as }|a| \rightarrow 1^{-}
\end{equation}
uniformly on compact subsets of $\mathbb{C}$. Given such a $\varphi$, a 
meromorphic function $f$ in $\mathbb{D}$ is said to be $\varphi$-normal if
$$
\sup _{z \in \mathbb{D}} \frac{f^{\#}(z)}{\varphi(|z|)}<\infty.
$$
Evidently, the class of $\varphi$-normal functions in 
$\mathbb{D}$ encompasses a broader set than the class of normal functions in $\mathbb{D}.$ We now introduce the definition of $\varphi$-normal harmonic mappings as introduced by Bohra et al. \cite{bohra}. 

\begin{definition}\cite[Definition 1.2]{bohra}
A harmonic mapping $f=h+\overline{g}$ in $\mathbb{D}$ is called $\varphi$-normal if 
\begin{equation}\label{eq: varphi h normal}
\sup_{z\in\mathbb{D}}\frac{f^{\#} (z)}{\varphi(|z|)}<\infty,
\end{equation}
 where $f^{\#} (z)$ is determined by \eqref{eq: spherical derivative}.
 \end{definition}
 It is apparent that the class of $\varphi$-normal harmonic mappings is broader then the class of normal harmonic mappings.

 \medskip

 Recall the well-known Lappan's five-point theorem \cite[Theorem~1]{lappan} which states that a meromorphic function 
$f$ is normal in $\mathbb{D}$ if there exists a set $E$ of five distinct points such that $\sup\{(1-|z|^2)f^\#(z): z\in f^{-1}(E)\}$ is bounded above.  Recently, Bohra et al. \cite[Theorem 1.5]{bohra} extended the Lappan's five-point theorem for $\varphi$-normal harmonic mappings as: a harmonic mapping $f=h+\bar{g}$ in $\mathbb{D}$ is normal 
if there is a set $E$ of five distinct complex number such that $$\sup\left\{\frac{f^\#(z)}{\varphi(|z|)}: z\in f^{-1}(E)\right\}<\infty.$$  
We extend this result in a more general setting as follows: 

\begin{theorem}\label{thm:lappan}
Let $f$ is a sense-preserving harmonic 
mapping in $\mathbb{D}$ and let $P$ be a non-constant polynomial of degree $m\geq 1$ such that the coefficients of $P$ are non-negative real numbers. If there exists a set E of five distinct complex numbers such that 
\begin{equation}\label{eq:lappan-1}
\sup\left\{\frac{P(f^{\#}(z))}{\varphi(|z|)}: z\in f^{-1}(E)\right\}<\infty,
\end{equation}
then $f$ is a $\varphi$-normal harmonic function.
\end{theorem}
Obviously, if we chose $P(z)=z$ in Theorem \ref{thm:lappan}, then Theorem \ref{thm:lappan} 
 coincides with the result of Bohra et al. \cite[Theorem 1.5]{bohra}. 

\medskip

Inspired by the idea of extended spherical derivative of a meromorphic functions by Lappan \cite[p. 307]{lappan-2}, we introduce the extended spherical derivative of a harmonic mapping as follows:\\
Let $k$ be a positive integer. Then the extended spherical derivative of a harmonic mapping $f=h+\bar{g}$ in $\mathbb{D}$ is given by 
\begin{equation}\label{eq:esp}
    f^{\#(k)}(z):=\frac{|h^{(k)}(z)|+|g^{(k)}(z)|}{1+|f(z)|^{k+1}}.
\end{equation}
While extended spherical derivatives of meromorphic functions have been extensively used in the study of normal families and normal meromorphic functions (see \cite{chen, lappan-2, xu, yang}), in this paper, we aim to apply the concept of the extended spherical derivative of a harmonic mapping to the study of 
 $\varphi$-normal harmonic mappings. 

\medskip

The following theorem provides a necessary condition for $\varphi$-normal harmonic mappings.

\begin{theorem}\label{thm:d}
 Let $\varphi:[0,1)\rightarrow (0,\infty)$ be a smoothly increasing function 
such that $1/\varphi$ is convex. Let $k$ be a positive integer and let $f$ be a harmonic mapping in $\mathbb{D}.$ If $f$ is $\varphi$-normal, then 
\begin{equation}\label{eq:d-1}
    \sup\left\{\frac{f^{\#(k)}(z)}{\varphi(|z|)^k}: z\in\mathbb{D}\right\}<\infty.
\end{equation}
\end{theorem}
In the following result, we employ the idea of extended spherical derivative to obtain a sufficient condition for a harmonic mapping in $\mathbb{D}$ to be $\varphi$-normal. 

\begin{theorem}\label{thm:y}
Let $k$ be a positive integer and let $f=h+\bar{g}$ be a sense-preserving harmonic mapping in $\mathbb{D}$ such that 
\begin{equation}\label{eq:y-1}
\sup\{f^{(i)}(z)=h^{(i)}(z)+\overline{g^{(i)}(z)}: z\in f^{-1}(\{0\}),~ i=0,\ldots, k-1\}<\infty.
\end{equation}
If there exists a set $E$ of $k+4$ distinct points in $\mathbb{C}$ such that 
\begin{equation}\label{eq:y-2}
\sup\left\{\frac{f^{\#(k)}(z)}{\varphi(|z|)^k}: z\in f^{-1}(E)\right\}<\infty,
\end{equation}
then $f$ is a $\varphi$-normal harmonic mapping.
\end{theorem}
Notice that if $f=h+\bar{g}$ has zeros of multiplicity at least $k,$ then the condition \eqref{eq:y-1} holds naturally. As an immediate consequence of Theorem \ref{thm:y}, we obtain the following result which is the  Lappan's five-point theorem for $\varphi$-normal harmonic maappings.
\begin{corollary}\cite[Theorem 1.5]{bohra}
Let $f=h+\bar{g}$ be a sense-preserving harmonic mapping in $\mathbb{D}.$ If there exists a set $E$ of five distinct complex numbers such that 
\begin{equation*}
\sup\left\{\frac{f^{\#}(z)}{\varphi(|z|)}: z\in f^{-1}(E)\right\}<\infty,
\end{equation*}
then $f$ is a $\varphi$-normal harmonic mapping.
\end{corollary}

Analogous to the result of Tan and the present second author \cite[Theorem 1]{tan} where the authors proved that the cardinality of the five-point set in Lappan's theorem \cite[Theorem 1]{lappan} can be reduced to four under suitable conditions. We reduce the cardinality of the set $E$ in Theorem \ref{thm:y} and obtain a generalization of \cite[Theorem 1.6]{bohra} as follows:
 
\begin{theorem}\label{thm:ya}
Let $k$ be a positive integer and let $f=h+\bar{g}$ be a sense-preserving harmonic mapping in $\mathbb{D}$ such that 
\begin{equation}\label{eq:ya-1}
\sup\{f^{(i)}(z)=h^{(i)}(z)+\overline{g^{(i)}(z)}: z\in f^{-1}(\{0\}),~ i=0,\ldots, k-1\}<\infty.
\end{equation}
If there exists a set $E$ of $[k/2]+4$ distinct points in $\mathbb{C}$ such that 
\begin{equation}\label{eq:ya-2}
\sup\left\{\frac{f^{\#(k)}(z)}{\varphi(|z|)^k}: z\in f^{-1}(E)\right\}<\infty,
\end{equation}
and 
\begin{equation}\label{eq:ya-3}
\sup\left\{\left(\frac{|h^{(k+1)}(z)|+|g^{(k+1)}(z)|}{1+(|h^{(k)}(z)|+|g^{(k)}(z)|)^{k+1}}\right): z\in f^{-1}(E)\right\}<\infty,
\end{equation}
then $f$ is a $\varphi$-normal harmonic mapping, where $[x]$ denotes the greatest integer less than or equal to $x.$
\end{theorem} 
By taking $k=1,$ in Theorem \ref{thm:ya}, we recover the following:
\begin{corollary}\cite[Theorem 1.6]{bohra}
   Let $f=h+\bar{g}$ be a sense preserving harmonic mapping in $\mathbb{D}.$ If there exists a set $E$ of four distinct complex numbers such that 
\begin{equation*}
\sup\left\{\frac{f^{\#}(z)}{\varphi(|z|)}: z\in f^{-1}(E)\right\}<\infty,
\end{equation*}
and 
\begin{equation*}
\sup\left\{\left(\frac{|h''(z)|+|g''(z)|}{1+(|h'(z)|+|g'(z)|)^{2}}\right): z\in f^{-1}(E)\right\}<\infty,
\end{equation*}
then $f$ is a $\varphi$-normal harmonic mapping.
\end{corollary}

\section{Preliminary Lemmas}
In this section, we state some preliminary results which are crucial to prove the main results of the present paper.

The following lemma generalizes Marty’s theorem \cite[Theorem 4]{marty}, extending it from families of meromorphic functions to families of harmonic mappings.

\begin{lemma}\label{lem:ponnu}\cite[Lemma 1]{deng}
A family $\mathcal{F}$ of harmonic mappings $f=h+\overline{g}$ in $\mathbb{D}$ is normal if 
$\{f^{\#}(z)\,:\,f\in \mathcal{F}\}$ is locally uniformly  bounded.  
\end{lemma}

The next lemma is a extension of a result of Lohwater and Pommerenke \cite[Theorem 1]{lohwater} to $\varphi$-normal harmonic mappings.

\begin{lemma}\cite[Theorem 1.3]{bohra}
A non-constant mapping $f$ harmonic in $\mathbb{D}$ is $\varphi$-normal if and only if there do not 
exist sequences $\{z_n\},~\{\rho_n\},~\rho_n>0,~\rho_n\to 0$ as $n\to\infty$ such that the functions $$g_n(\zeta)=f\left(z_n+\frac{\rho_n\zeta}{\varphi(|z_n|)}\right)\to g(\zeta)$$ locally uniformly in $\mathbb{C},$ where $g$ is a non-constant harmonic mapping.
\end{lemma}

The following necessary condition for $\varphi$-normal harmonic mappings is useful to our study.

\begin{lemma}\label{lem:bora}\cite[Theorem 1.4]{bohra}
Let $f$ be a harmonic mapping in $\mathbb{D}$ and let $\varphi:[0,1)\rightarrow(0,\infty)$ be a smoothly increasing function 
such that $1/\varphi$ is convex. If $f$ is $\varphi$-normal, then, 
for any sequence of points $\{z_n\}\subset\mathbb{D}$ with $z_n\to 1^-$ as $n\to\infty,$ the family $\{f(z_n+z/\varphi(|z_n|))\}$ is normal in $\mathbb{C}.$  
\end{lemma}


What  follows is a generalization of Hurwitz's theorem for harmonic mappings.
\begin{lemma}\cite[p. 10]{duren}\label{lem:hurwitz}
Let $\{f_n\}$ be a sequence of sense-preserving harmonic mappings in $\mathbb{D}$ such that $f_n$ converges
locally uniformly to a sense-preserving harmonic mapping $f$. Then  $z_0\in\mathbb{D}$ is a zero of $f$ if and 
only if $z_0$ is a cluster point of the zeros of $f_n$, $n\geq 1$.
\end{lemma}
We also need the following result in the proof of Theorem \ref{thm:lappan}.

\begin{lemma}\cite[Lemma 3]{deng}\label{lem:deng}
Let $f=h+\overline{g}$ be a sense-preserving harmonic mapping in $\mathbb{C}$ with $g(0)=0$. Then there are  
at most four values of $a$ for which all zeros of $f-a$ are multiple. 
\end{lemma}

Finally, we need some results from the value distribution theory of Nevanlinna. We assume that the standard results and notations of the Nevanlinna's theory such as $m(r, f),$ $N(r, f),$ $T(r, f),$ and $S(r, f)$ is known to the reader. However, for the sake of completion, we state the required results below. For details, we refer the reader to \cite{hayman}

\begin{lemma}[{\bf First fundamental theorem}]\label{lem:n-1}
    Let $f$ be a meromorphic function in $\mathbb{C}.$ Then for any $a\in \mathbb{C},$ we have $$T\left(r,\frac{1}{f-a}\right)=T(r, f) + O(1).$$
\end{lemma}

\begin{lemma}[{\bf Second fundamental theorem}]\label{lem:n-2}
    Let $f$ be a meromorphic function in $\mathbb{C}$ and $a_i,~1\leq i\leq q$ be $q~(\geq 3)$ distinct values in $\overline{\mathbb{C}}.$ Then $$(q-2)T(r, f)\leq\sum_{i=1}^{q}\overline{N}\left(r, \frac{1}{f-a_i}\right)+ S(r, f).$$
\end{lemma}

\begin{lemma}\label{lem:n-3}
    Let $k$ be a positive integer and $f$ be a non-constant meromorphic function in $\mathbb{C}.$ Then $$T(r, f^{(k)})\leq (k+1)T(r, f)+ S(r, f).$$
\end{lemma}

\section{Proof of Main Results}

\begin{proof}[Proof of Theorem \ref{thm:zp}]
Suppose that $f$ is not normal. Then there exists a sequence $\{z_n^*\}\subset\mathbb{D}$ with $z_n^*\to 1$ as $n\to\infty$ such that
\begin{equation}\label{eq:zp-1}
  (1-|z_n^*|^2)f^{\#}(z_n^*)\to\infty \mbox{ as } n\to\infty.  
\end{equation}
Let $\{r_n\}$ be a sequence such that $|z_n^*|<r_n<1$ and  
\begin{equation}\label{eq:zp-2}
  \left(1-\frac{|z_n^*|^2}{r_n^2}\right)f^{\#}(z_n^*)\to\infty \mbox{ as } n\to\infty.  
\end{equation}
Define $$F_n(t,z):=\frac{\left(1-\frac{|z|^2}{r_n^2}\right)^{1+\alpha}t^{1+\alpha}(1+ |f(z)|^2)f^{\#}(z)}{1+\left(1-\frac{|z|^2}{r_n^2}\right)^{2\alpha}t^{2\alpha}|f(z)|^2},$$ where $0<t\leq 1$ and $|z|<r_n.$
Clearly, the functions $F_n(t,z)$ are continuous in $(0,1]\times\{|z|<r_n\}.$ Also, since $\alpha>-1$ and $f$ is harmonic, it follows that 
\begin{equation}\label{eq:zp-3}
    \lim\limits_{t\to 0}F_n(t,z)=0.
\end{equation}
We claim that \begin{equation}\label{eq:zp-4}
   F_n(t,z)\geq t^{1+|\alpha|}\left(1-\frac{|z|^2}{r_n^2}\right)^{1+|\alpha|}f^{\#}(z).
\end{equation}
Indeed, if $\alpha>0,$ then using $\left(1-\frac{|z|^2}{r_n^2}\right)^{2\alpha}t^{2\alpha}\leq 1,$ we have
\begin{align*}
F_n(t,z) &=\frac{\left(1-\frac{|z|^2}{r_n^2}\right)^{1+\alpha}t^{1+\alpha}(1+ |f(z)|^2)f^{\#}(z)}{1+\left(1-\frac{|z|^2}{r_n^2}\right)^{2\alpha}t^{2\alpha}|f(z)|^2}\\
&\geq \frac{\left(1-\frac{|z|^2}{r_n^2}\right)^{1+\alpha}t^{1+\alpha}(1+ |f(z)|^2)f^{\#}(z)}{1+|f(z)|^2}\\
&=t^{1+\alpha}\left(1-\frac{|z|^2}{r_n^2}\right)^{1+\alpha}f^{\#}(z),
\end{align*}
and if $\alpha<0,$ then  using $\left(1-\frac{|z|^2}{r_n^2}\right)^{-2\alpha}t^{-2\alpha}\leq 1,$ we have
\begin{align*}
F_n(t,z) &=\frac{\left(1-\frac{|z|^2}{r_n^2}\right)^{1-\alpha}t^{1-\alpha}(1+ |f(z)|^2)f^{\#}(z)}{1+\left(1-\frac{|z|^2}{r_n^2}\right)^{-2\alpha}t^{-2\alpha}|f(z)|^2}\\
&\geq \frac{\left(1-\frac{|z|^2}{r_n^2}\right)^{1-\alpha}t^{1-\alpha}(1+ |f(z)|^2)f^{\#}(z)}{1+|f(z)|^2}\\
&=t^{1-\alpha}\left(1-\frac{|z|^2}{r_n^2}\right)^{1-\alpha}f^{\#}(z).
\end{align*}
This proves the claim. Now, from \eqref{eq:zp-2} and \eqref{eq:zp-4}, we get $$F_n(1, z_n^*)\geq \left(1-\frac{|z_n^*|^2}{r_n^2}\right)^{1+|\alpha|}f^{\#}(z_n^*)\to\infty \mbox{ as } n\to\infty.$$
Therefore, for sufficiently large $n,$ $F_n(1, z_n^*)>1$ and hence $$\sup\limits_{|z|<r_n}F_n(1,z)\geq F_n(1, z_n^*)>1.$$
On the other hand, from \eqref{eq:zp-2}, it follows that $$\sup\limits_{|z|<r_n}F_n(t,z)<1 \mbox{ for sufficiently small } t.$$ Thus, by the intermediate value theorem, there exist $t_n\in (0,1)$ and $z_n$ with $|z_n|<r_n$ such that
\begin{equation}\label{eq:zp-5}
    \sup\limits_{|z|<r_n}F_n(t_n,z)=F_n(t_n, z_n)=1.
\end{equation}
Now, from \eqref{eq:zp-4} and \eqref{eq:zp-5}, we have
$$1=F_n(t_n, z_n)\geq F_n(t_n, z_n^*)\geq t_n^{1+|\alpha|}\left(1-\frac{|z_n^*|^2}{r_n^2}\right)^{1+|\alpha|}f^{\#}(z_n^*)=t_n^{1+|\alpha|}F_n(1, z_n^*).$$
Since $F_n(1, z_n^*)\to\infty$ as $n\to\infty,$ it follows that $\lim\limits_{n\to\infty}t_n=0.$

Put $\rho_n:=\left(1-\frac{|z_n|^2}{r_n^2}\right)t_n.$ Then $\rho_n\to\infty$ as $n\to\infty.$ Also,
\begin{equation}\label{eq:rho}
\frac{\rho_n}{1-|z_n|}\leq \frac{\rho_n}{r_n-|z_n|}=\frac{(r_n+|z_n|)t_n}{r_n^2}\to 0 \mbox{ as } n\to\infty.
\end{equation}
Thus the function $g_n(\zeta)=\rho_n^{\alpha} f(z_n+\rho_n\zeta)$ is defined for $|\zeta|<R_n=\frac{1-|z_n|}{\rho_n}.$ Furthermore, a simple computation shows that
\begin{equation}\label{eq:zp-6}
    g_n^{\#}(\zeta)=\frac{\rho_n^{1+\alpha}(1+|f(z_n+\rho_n\zeta)|^2)f^{\#}(z_n+\rho_n\zeta)}{1+\rho_n^{2\alpha}|f(z_n+\rho_n\zeta)|^2}.
\end{equation}
Then 
\begin{align}\label{eq:zp-7}
g_n^{\#}(0)&=\frac{\rho_n^{1+\alpha}(1+|f(z_n)|^2)f^{\#}(z_n)}{1+\rho_n^{2\alpha}|f(z_n)|^2}\nonumber\\
&=\frac{\left(1-\frac{|z_n|^2}{r_n^2}\right)^{1+\alpha}t_n^{1+\alpha}(1+ |f(z_n)|^2)f^{\#}(z_n)}{1+\left(1-\frac{|z_n|^2}{r_n^2}\right)^{2\alpha}t_n^{2\alpha}|f(z_n)|^2}\nonumber\\
&=F_n(t_n, z_n)=1.
\end{align}
We now turn to show that the sequence $\{g_n(\zeta)\}$ is normal.

Since $\left(1-\frac{|z_n|^2}{r_n^2}\right)/\left(1-\frac{|z_n+\rho_n\zeta|^2}{r_n^2}\right)\to 1$ as $n\to\infty,$ there exists $\epsilon_n>0$ with $\epsilon_n\to 0$ as $n\to\infty$ such that
\begin{equation}\label{eq:zp-8}
    (1-\epsilon_n)\left(1-\frac{|z_n+\rho_n\zeta|^2}{r_n^2}\right)t_n\leq\rho_n\leq (1+\epsilon_n)\left(1-\frac{|z_n+\rho_n\zeta|^2}{r_n^2}\right)t_n.
\end{equation}
Therefore, from \eqref{eq:zp-5}, \eqref{eq:zp-6} and \eqref{eq:zp-8}, we obtain
\begin{align}\label{eq:zp-9}
  g_n^{\#}(\zeta)&\leq \frac{(1+\epsilon_n)^{1+\alpha}\left(1-\frac{|z_n+\rho_n\zeta|^2}{r_n^2}\right)^{1+\alpha}t_n^{1+\alpha}(1+ |f(z_n+\rho_n\zeta)|^2)f^{\#}(z_n+\rho_n\zeta)}{1+(1-\epsilon_n)^{2\alpha}\left(1-\frac{|z_n+\rho_n\zeta|^2}{r_n^2}\right)^{2\alpha}t_n^{2\alpha}|f(z_n+\rho_n\zeta)|^2}\nonumber\\
  &\leq \frac{(1+\epsilon_n)^{1+\alpha}\left(1-\frac{|z_n+\rho_n\zeta|^2}{r_n^2}\right)^{1+\alpha}t_n^{1+\alpha}(1+ |f(z_n+\rho_n\zeta)|^2)f^{\#}(z_n+\rho_n\zeta)}{(1-\epsilon_n)^{2\alpha}\left(1-\frac{|z_n+\rho_n\zeta|^2}{r_n^2}\right)^{2\alpha}t_n^{2\alpha}|f(z_n+\rho_n\zeta)|^2}\nonumber\\
  &\leq \frac{(1+\epsilon_n)^{1+\alpha}}{(1-\epsilon_n)^{2\alpha}}\to 1 \mbox{ as } n\to\infty.
\end{align}
Thus, by Lemma \ref{lem:ponnu}, we conclude that the sequence $\{g_n(\zeta)\}$ is normal and so we may assume that $\{g_n(\zeta)\}$ converges locally uniformly in $\mathbb{C}$ to a harmonic mapping $g(\zeta).$ Then from \eqref{eq:zp-7} and \eqref{eq:zp-9}, it follows that $g^{\#}(\zeta)\leq g^{\#}(0)=1\neq 0$ which further implies that $g$ is a non-constant harmonic mapping. To see the converse part, let $\alpha=0.$ Then 
 from Corollary \ref{cor:deng}, we deduce that $f$ is not normal. This completes the proof.
\end{proof}
\begin{proof}[Proof of Theorem \ref{thm:g}]
 Suppose, on the contrary, that $f$ is not normal. Then there exists a sequence $\{z_n\}\subset\mathbb{D}$ and a sequence of positive real numbers $\{\rho_n\}$ such that $z_n\to 1^{-},~\rho_n\to 0$ as $n\to\infty,$ and $$F_n(\zeta):=\rho_n^2 f(z_n+\rho_n\zeta)=\rho_n^2\left[h((z_n+\rho_n\zeta)+ \overline{g((z_n+\rho_n\zeta)}\right]\to F(\zeta)$$ locally uniformly in $\mathbb{C},$ where $F$ is a non-constant harmonic mapping  such that $F^{\#}(\zeta)\leq 1$ on $\mathbb{C}$ and $F^{\#}(0)=1.$ Then there exists $a\in\mathbb{D}$ such that $|F(a)|>0,$ and therefore, for sufficiently large $n,$ $|F_n(a)|\neq 0.$ Thus, we have 
 \begin{align*}
   1\geq F_n^{\#}(a)&= \frac{\rho_n^3\left(|h'(z_n+\rho_n a)|+|g'((z_n+\rho_n a)|\right)}{1+\rho_n^4|f((z_n+\rho_n a)|^2}\\
   &=\frac{\rho_n^3\left(|h'(z_n+\rho_n a)|+|g'((z_n+\rho_n a)|\right)}{\rho_n^4|f((z_n+\rho_n a)|^2}\times\frac{|F_n(a)|^2}{1+|F_n(a)|^2}\\
   &\geq \frac{f^{\#}(z_n+\rho_n a)}{\rho_n}\times\frac{|F_n(a)|^2}{1+|F_n(a)|^2}\\
   &\geq \frac{\epsilon}{\rho_n}\times\frac{|F_n(a)|^2}{1+|F_n(a)|^2}\\
   &\to\infty \mbox{ as } n\to\infty.
 \end{align*}
This is a contradiction to the fact that $F^{\#}(\zeta)\leq 1,$ for every $\zeta\in\mathbb{C}.$ Thus, $f$ is a normal harmonic mapping.
\end{proof}
\begin{proof}[Proof of Theorem \ref{thm:lappan}]
Suppose that $f$ is not $\varphi$-normal harmonic mapping in $\mathbb{D}.$ Then
by Lemma \ref{lem:bora}, there exist sequences $\{z_n\}, \ \{\rho_n\}$, $\rho_n>0$ with $z_n\to 1^-$ and $\rho_n\to 0$ as $n\to\infty$ such that 
$$F_n(\zeta)=f\left(z_n+\frac{\rho_n}{\varphi(|z_n|)}\zeta\right)\to F(\zeta)$$ locally uniformly in $\mathbb{C},$ where $F$ is a non-constant sense-preserving harmonic mapping. Let $a\in {\mathbb{C}}$ such that 
$F-a$ has a zero at $\zeta_0$ which is not a multiple zero. This implies that $F^{\#}(\zeta_0)\neq 0.$ By Lemma~\ref{lem:hurwitz}, there exists $\zeta_n\to\zeta_0$ such that for sufficiently large $n,$ 
$$F_n(\zeta_n)=f\left(z_n+\frac{\rho_n}{\varphi(|z_n|)}\zeta_n\right)=a.$$ Put $w_n=\left(z_n+\frac{\rho_n}{\varphi(|z_n|)}\zeta_n\right)$ so that $F_n(\zeta_n)=f(w_n)=a.$ Since $F_n\to F$ locally uniformly, it follows that $$F_n^{\#}(\zeta_n)=\frac{\rho_n}{\varphi(|z_n|)}f^{\#}(w_n)\to F^{\#}(\zeta_0) \text{ as } n\to\infty$$ and hence, for each $t\geq 1,$ we have  $$\left(F_n^{\#}(\zeta_n)\right)^t=\frac{\rho_n^t}{\varphi(|z_n|)^t}\left(f^{\#}(w_n)\right)^t\to \left(F^{\#}(\zeta_0)\right)^t \text{ as } n\to\infty.$$  
Then 
\begin{align*}
\frac{\left(f^{\#}(w_n)\right)^t}{\varphi(|w_n|)} &=\frac{\varphi(|z_n|)^t}{\rho_n^t}\frac{\left(F_n^\#(\zeta_n)\right)^t}{\varphi(|w_n|)}\\
&=\frac{\varphi(|z_n|)^{t-1}}{\rho_n^t}\left(F^\#(\zeta_n)\right)^t\frac{\varphi(|z_n|)}{\varphi(|w_n|)}\\
&\geq \frac{1}{\rho_n^t(1-|z_n|)^{t-1}}\left(F^\#(\zeta_n)\right)^t\frac{\varphi(|z_n|)}{\varphi(|w_n|)}\\
&\to\infty \mbox{ as } n\to\infty,
\end{align*}
 because $\left(F_n^\#(\zeta_n)\right)^t\to \left(F^\#(\zeta_0)\right)^t\neq 0,~ \varphi(|w_n|)/\varphi(|z_n|)\to 1$, and $\rho_n\to 0$ as $n\to\infty.$

 Let $P(z)=a_mz^m+ a_{m-1}z^{m-1}+\cdots+a_1z+a_0,$ where $a_i,~0\leq i\leq m$ are non-negative real numbers and $a_m\neq 0.$ Then $$\frac{P(f^\#(w_n))}{\varphi(|w_n)}=a_m\frac{\left(f^\#(w_n)\right)^m}{\varphi(|w_n)}+ a_{m-1}\frac{\left(f^\#(w_n)\right)^{m-1}}{\varphi(|w_n)}+\cdots+ a_1\frac{\left(f^\#(w_n)\right)}{\varphi(|w_n)}+a_0\to\infty \mbox{ as } n\to\infty.$$
To this end, we have established that if $F(\zeta)-a$ has a simple zero, then $$\sup_{z\in f^{-1}(E)}\frac{P(f^{\#}(z))}{\varphi(|z|)}=\infty.$$

On the other hand, Lemma \ref{lem:deng} asserts that there are no more than four values of $a$ for which all zeros of $F(\zeta)-a$ are multiple zeros. Thus, we infer that if $f$ is a sense-preserving harmonic mapping which is not $\varphi$-normal in $\mathbb{D},$ then for any complex number $a$ with at most four possible exceptions, we have
$$\sup_{z\in f^{-1}(E)}\frac{P(f^{\#}(z))}{\varphi(|z|)}<\infty.$$
This completes the proof.
\end{proof}
\begin{proof}[Proof of Theorem \ref{thm:d}]
 For $k=1,$ the result follows by definition of $\varphi$-normality. So, let $k\geq 2$ and suppose that $f$ is $\varphi$-normal but does not satisfy \eqref{eq:d-1}. Then there exists a sequence $\{z_n\}\subset\mathbb{D}$ with $z_n\to 1^-$ as $n\to\infty$ such that 
 \begin{equation}\label{eq:d-2}
   \frac{f^{\#(k)}(z_n)}{\varphi(|z_n|)^k} \to\infty \mbox{ as } n\to\infty.
\end{equation}
Define $$\mathcal{F}=\left\{F_n(z):=f\left(z_n+\frac{z}{\varphi(|z_n|)}\right)\right\}.$$ Then by Lemma \ref{lem:bora}, $\mathcal{F}$ is a normal family in $\mathbb{C}$ and so for each sequence $\{F_n\}\subset\mathcal{F},$ there is a subsequence of $\{F_n\},$ which we shall again denote by $\{F_n\},$ such that $\{F_n\}$ converges locally uniformly in $\mathbb{C}$ to a harmonic mapping $F=H+\overline{G}.$ This implies that $F_n^{\#(k)}(z)\to F^{\#(k)}(z)$ as $n\to\infty.$ That is, $$\frac{1}{\varphi(|z_n|)^k}f^{\#(k)}\left(z_n+\frac{z}{\varphi(|z_n|)}\right)\to F^{\#(k)}(z)=\frac{|H^{(k)}(z)|+|G^{(k)}(z)}{1+|F(z)|^{k+1}} \mbox{ as } n\to\infty.$$ Letting $z=0,$ we get $$\frac{f^{\#(k)}(z_n)}{\varphi(|z_n|)^k}\to F^{\#(k)}(0)=\frac{|H^{(k)}(0)|+|G^{(k)}(0)}{1+|F(0)|^{k+1}} \mbox{ as } n\to\infty.$$
Take $M=F^{\#(k)}(0).$ Then, for sufficiently large $n,$ $$\frac{f^{\#(k)}(z_n)}{\varphi(|z_n|)^k}\leq M+1,$$ a contradiction to \eqref{eq:d-2}. This completes the proof.
\end{proof}

\begin{proof}[Proof of Theorem \ref{thm:y}]
Suppose, on the contrary, that $f=h+\bar{g}$ is not $\varphi$-normal. Then there exists a sequence of points$\{z_n\}\subset \mathbb{D}$ and positive real numbers $\rho_n$ with 
$z_n\to 1^-$ and $\rho_n\to 0$ as $n\to\infty$ such that the sequence 
\begin{equation}\label{eq:y-3} 
F_n(\zeta):=f\left(z_n+\frac{\rho_n\zeta}{\varphi(|z_n|)}\right)=h\left(z_n+\frac{\rho_n\zeta}{\varphi(|z_n|)}\right)+
\overline{g\left(z_n+\frac{\rho_n\zeta}{\varphi(|z_n|)}\right)}
\end{equation}
converges locally uniformly in $\mathbb{C}$ to a non-constant harmonic mapping $F(\zeta)=H(\zeta)+\overline{G(\zeta)}.$ 
Let $F_n(\zeta)=h_n(\zeta)+\overline{g_n(\zeta)},$ where  $$h_n(\zeta):= h\left(z_n+\frac{\rho_n\zeta}{\varphi(|z_n|)}\right) \text{ and }
g_n(\zeta):= g\left(z_n+\frac{\rho_n\zeta}{\varphi(|z_n|)}\right).$$ 
Then $h_n^{(i)}(\zeta)\to H^{(i)}(\zeta)$ and $g_n^{(i)}(\zeta)\to G^{(i)}(\zeta)$ as $n\to\infty.$ Therefore, 
$$F_n^{(i)}(\zeta)=\frac{\rho_n^i}{\varphi(|z_n|)^i}f^{(i)}\left(z_n+\frac{\rho_n\zeta}{\varphi(|z_n|)}\right)\to F^{(i)}(\zeta)=H^{(i)}(\zeta)+\overline{G^{(i)}(\zeta)}.$$\\ 
{\bf Claim 1:} Each zero of $F$ has multiplicity at least $k.$\\
Let $\zeta_0\in\mathbb{C}$ be such that $F(\zeta_0)=0.$ Then by Lemma \ref{lem:hurwitz}, there exists a sequence 
$\{\zeta_n\}$ with $\zeta_n\to\zeta_0$ as $n\to\infty$ such that 
\begin{equation*}
F_n(\zeta_n)=f\left(z_n+\frac{\rho_n\zeta_n}{\varphi(|z_n|)}\right)=0
\end{equation*}
By \eqref{eq:y-1}, there exists an $M>0$ such that 
\begin{equation*}
f^{(i)}\left(z_n+\frac{\rho_n\zeta}{\varphi(|z_n|)}\right)\leq M,~i=0,~\ldots,~k-1
\end{equation*}
Then 
\begin{equation}\label{eq:y-4}
F^{(i)}(\zeta_0)=\lim\limits_{n\to\infty} F_n^{(i)}(\zeta_n)=\lim\limits_{n\to\infty}\frac{\rho_n^i}{\varphi(|z_n|)^i}f^{(i)}\left(z_n+\frac{\rho_n\zeta}{\varphi(|z_n|)}\right)\leq\lim\limits_{n\to\infty}M\frac{\rho_n^i}{\varphi(|z_n|)^i}
\end{equation}
Using \eqref{eq:phi} and \eqref{eq:rho} in \eqref{eq:y-4}, we obtain $$F^{(i)}(\zeta_0)\leq\lim\limits_{n\to\infty}M\frac{\rho_n^i}{\varphi(|z_n|)^i}\leq\lim\limits_{n\to\infty}M\left(\frac{\rho_n}{1-|z_n|}\right)^i=0,~i=0,~\ldots,~k-1$$ 
This proves Claim $1.$\\
{\bf Claim 2:} $F(\zeta)=H(\zeta)+\overline{G(\zeta)}\in E\Rightarrow H^{(k)}(\zeta)=0$ and $G^{(k)}(\zeta)=0$\\
For any $a\in E,$ let $\zeta_0^*$ be a zero of $F-a.$ Then, again by Lemma \ref{lem:hurwitz}, there exists a sequence 
$\{\zeta_n^*\}$ with $\zeta_n^*\to\zeta_0^*$ as $n\to\infty$ such that 
\begin{equation*}
F_n(\zeta_n^*)=f\left(z_n+\frac{\rho_n\zeta_n^*}{\varphi(|z_n|)}\right)=a.
\end{equation*}
By \eqref{eq:y-2}, there exists $M^*>0$ such that $$f^{\#(k)}\left(z_n+\frac{\rho_n\zeta_n^*}{\varphi(|z_n|)}\right)\leq M^*\varphi\left(\left|z_n+\frac{\rho_n\zeta_n^*}{\varphi(|z_n|)}\right|\right)^k.$$
Since $$F_n^{(k)}(\zeta_n^*)=\frac{\rho_n^k}{\varphi(|z_n|)^k}f^{(k)}\left(z_n+\frac{\rho_n\zeta_n^*}{\varphi(|z_n|)}\right),$$ we have
\begin{align*}
 F_n^{\#(k)}(\zeta_n^*)&=  \frac{\rho_n^k}{\varphi(|z_n|)^k}f^{\#(k)}\left(z_n+\frac{\rho_n\zeta_n^*}{\varphi(|z_n|)}\right)\\
 &\leq \rho_n^k M^*\left[\frac{\varphi\left(\left|z_n+\frac{\rho_n\zeta_n^*}{\varphi(|z_n|)}\right|\right)}{\varphi(|z_n|)}\right]^k\\
 &\to 0 \mbox{ as } n\to\infty.
\end{align*}
Thus, $F^{\#(k)}(\zeta_0^*)=\lim\limits_{n\to\infty}F_n^{\#(k)}(\zeta_n^*)=0,$ that is, $$\frac{|H^{(k)}(\zeta_0^*)|+|G^{(k)}(\zeta_0^*)|}{1+|F(\zeta_0^*)|^{k+1}}=0.$$
Since $F$ is non-constant harmonic mapping, it follows that $H^{(k)}(\zeta_0^*)=0$ and $G^{(k)}(\zeta_0^*)=0.$ This proves Claim $2.$

Now, since $F(\zeta)=H(\zeta)+\overline{G(\zeta)}$ is a sense preserving harmonic mapping in $\mathbb{C},$ therefore, $|G'(\zeta)/H'(\zeta)|<1$ in $\mathbb{C},$ and so by Liouville's theorem, $\omega(\zeta)=G'(\zeta)/H'(\zeta)=b,$ where $b\in\mathbb{C}$ with $|b|<1.$ As in the proof of \cite[Lemma 2.5]{bohra}, for any $a\in\mathbb{C},$ $F(\zeta)=a$ is equivalent to 
\begin{equation}\label{eq:y-5}
H(\zeta)=\frac{a-\overline{ba}+\overline{bH(0)}-|b|^2H(0)}{1-|b|^2}.
\end{equation}
In particular, if $a\in E,$ then $F(\zeta)=a$ is equivalent to $H(\zeta)=a^*,$ where $a^*$ is determined by \eqref{eq:y-5}. Thus, corresponding to each $a\in E$ with $F(\zeta)\in E,$ there is a set $E^*$ consisting of points of the form $a^*$ such that $H(\zeta)\in E^*.$ This together with Claim $2$ imply that $H^{(k)}(\zeta)=0$ whenever $H(\zeta)\in E^*.$ Again, since $F$ is a non-constant and sense-preserving harmonic mapping in $\mathbb{C}$ having zeros of multiplicity at least $k,$ it follows that $H$ is a non-constant entire function having zeros of multiplicity at least $k$ and therefore, $H^{(k)}\not\equiv 0.$ Now, Lemma \ref{lem:n-2} along with Lemmas \ref{lem:n-1} and \ref{lem:n-3} yield 
\begin{align*}
 (k+2)T(r, H) &\leq \sum\limits_{a^*\in E^*}\overline{N}\left(r, \frac{1}{H-a^*}\right) + S(r, H)\\
 &\leq \overline{N}\left(r, \frac{1}{H^{(k)}}\right) + S(r, H)\\
 &\leq T(r, H^{(k)}) + S(r, H)\\
 &\leq (k+1)T(r, H) + S(r, H),
\end{align*}
which is a contradiction. Hence $f$ is a $\varphi$-normal harmonic mapping.
\end{proof}
\begin{proof}[Proof of Theorem \ref{thm:ya}]
    Following the approach and notations used in the proof of the Theorem \ref{thm:y}, we have

    \begin{itemize}
        \item[(i)] A sequence $F_n(\zeta)=h_n(\zeta)+\overline{g_n(\zeta)}\to F(\zeta)=H(\zeta)+\overline{G(\zeta)}$ locally uniformly in $\mathbb{C},$ where $$F_n(\zeta)=f\left(z_n+\frac{\rho_n\zeta}{\varphi(|z_n|)}\right),~h_n(\zeta)= h\left(z_n+\frac{\rho_n\zeta}{\varphi(|z_n|)}\right),$$ and $$
g_n(\zeta)= g\left(z_n+\frac{\rho_n\zeta}{\varphi(|z_n|)}\right).$$
\item[(ii)] Each zero of $F$ has multiplicity at least $k.$ 
\item[(iii)] $F(\zeta)\in E\Rightarrow H^{(k)}(\zeta)=0$ and $G^{(k)}(\zeta)=0.$ 
    \end{itemize}
Now, we claim that $F(\zeta)\in E\Rightarrow H^{(k+1)}(\zeta)=0.$ Indeed, if $\zeta_0$ is a zero of $F-a,$ for any $a\in E,$ then there exists a sequence $\zeta_n\to\zeta_0$ such that for sufficiently large $n,$ $$F_n(\zeta_n)=f(w_n)=a, \text{ where } w_n:= \left(z_n+\frac{\rho_n\zeta_n}{\varphi(|z_n|)}\right).$$ By the condition $\eqref{eq:ya-2},$ there exists $M_1>0$ such that $$f^{\#(k)}(w_n)\leq M_1\varphi(w_n).$$ This implies that
\begin{align}\label{eq:ya-4}
  \left|h^{(k)}(w_n)\right| + \left|g^{(k)}(w_n)\right| &\leq M_1\left(1+\left|f(w_n)\right|^{k+1}\right)\varphi(w_n)\nonumber\\
  &\leq M_1\left(1+\max_{a\in E}\left|a\right|^{k+1}\right)\varphi(w_n).
\end{align}
Furthermore, by the condition \eqref{eq:ya-3}, there exists $M_2>0$ such that
\begin{equation}\label{eq:ya-5}
   \frac{\left|h^{(k+1)}(w_n)\right| + \left|g^{(k+1)}(w_n)\right|}{1+\left( \left|h^{(k)}(w_n)\right| + \left|g^{(k)}(w_n)\right|\right)^{k+1}} \leq M_2.
\end{equation}
Let $M=\max\{M_1,~M_2\}.$ Then from \eqref{eq:ya-4} and \eqref{eq:ya-5}, we have 
\begin{align*}
 \frac{\left|h_n^{(k+1)}(\zeta_n)\right| + \left|g_n^{(k+1)}(\zeta_n)\right|}{1+\left( \left|h_n^{(k)}(\zeta_n)\right| + \left|g_n^{(k)}(\zeta_n)\right|\right)^{k+1}} &=\frac{\left(\frac{\rho_n}{\varphi(|z_n|)}\right)^{k+1}\left(|h^{(k+1)}(w_n)|+|g^{(k+1)}(w_n)|\right)}{1+\left(\frac{\rho_n}{\varphi(|z_n|)}\right)^{k(k+1)}\left(|h^{(k)}(w_n)|+|g^{(k)}(w_n)|\right)^{k+1}}\\
 &= \frac{\left(\frac{\rho_n}{\varphi(|z_n|)}\right)^{k+1}\left(|h^{(k+1)}(w_n)|+|g^{(k+1)}(w_n)|\right)}{1+ \left(|h^{(k)}(w_n)|+|g^{(k)}(w_n)|\right)^{k+1}}\\ &\qquad
 \times \frac{1+ \left(|h^{(k)}(w_n)|+|g^{(k)}(w_n)|\right)^{k+1}}{1+\left(\frac{\rho_n}{\varphi(|z_n|)}\right)^{k(k+1)}\left(|h^{(k)}(w_n)|+|g^{(k)}(w_n)|\right)^{k+1}}\\
 &\leq M\left(\frac{\rho_n}{\varphi(|z_n|)}\right)^{k+1}\left(1+\left(|h^{(k)}(w_n)|+|g^{(k)}(w_n)|\right)\right)^{k+1}\\
 &\leq \rho_n^{k+1}M\left(1+\left(M\left(1+\max_{a\in E}|a|^{k+1}\right)\right)^{k+1}\right)\left(\frac{\varphi(|w_n|)}{\varphi(|z_n|)}\right)^{k+1}\\
 &\to 0 \mbox{ as } n\to\infty.
\end{align*}
Therefore, $|H^{(k+1)}(\zeta_0)|+|G^{(k+1)}(\zeta_0)|=0$ and hence $H^{(k+1)}(\zeta_0)=G^{(k+1)}(\zeta_0)=0.$ Proceeding as in the proof of Theorem \ref{thm:y}, we find that for any $a\in E,$ $F(\zeta)=a$ is equivalent to $H(\zeta)=a^*,$ where $a^*$ is determined by \eqref{eq:y-5}. Thus, corresponding to the set $E,$ we get a set $E^*$ such that $H(\zeta)\in E^*\Rightarrow H^{(i)}(\zeta)=0,~i=k,~k+1.$  
Now, Lemma \ref{lem:n-2} along with Lemmas \ref{lem:n-1} and \ref{lem:n-3} yield 
\begin{align*}
 ([k/2]+2)T(r, H) &\leq \sum\limits_{a^*\in E^*}\overline{N}\left(r, \frac{1}{H-a^*}\right) + S(r, H)\\
 &\leq \overline{N}_{(2}\left(r, \frac{1}{H^{(k)}}\right) + S(r, H)\\
 &\leq \frac{1}{2}N\left(r, \frac{1}{H^{(k)}}\right) + S(r, H)\\
 &\leq \frac{1}{2}T(r, H^{(k)}) + S(r, H)\\
 &\leq \frac{k+1}{2}T(r, H) + S(r, H)\\
 &\leq (k/2+1)T(r, H) + S(r, H).
\end{align*}
which is a contradiction. Hence $f$ is a $\varphi$-normal harmonic mapping.
\end{proof}

\medskip


\end{document}